\documentclass[11pt, reqno]{amsart}
\usepackage{graphicx, amssymb, amsmath, amsthm}
\usepackage[utf8]{inputenc}
\usepackage{epsfig}
\usepackage{hyperref}
\usepackage{comment}
\usepackage{mathrsfs}
\numberwithin{equation}{section}
\usepackage{amsthm}
\usepackage{subfigure}
\usepackage{tikz}
\usepackage{here}
\usepackage{float}
\usepackage{pifont}
\usepackage{enumitem}
\usetikzlibrary{matrix,arrows}
\usetikzlibrary{shapes}
\usetikzlibrary{calc}
\usetikzlibrary{arrows}
\usetikzlibrary{decorations.pathreplacing,decorations.markings}
\usepackage[all]{xy}
\usepackage[backend=biber, style=numeric]{biblatex} 
\usepackage{tikz-cd}

\usetikzlibrary{patterns}

\newtheorem{theorem}{Theorem}[section]

\newtheorem{lemma}[theorem]{Lemma}

\newtheorem{corollary}[theorem]{Corollary}
\newtheorem{conjecture}[theorem]{Conjecture}
\newtheorem{proposition}[theorem]{Proposition}

\theoremstyle{definition}
\newtheorem{definition}[theorem]{Definition}
\newtheorem{remark}[theorem]{Remark}

\makeatletter
\newcommand{\Extend}[5]{\ext@arrow0099{\arrowfill@#1#2#3}{#4}{#5}}
\makeatother

\DeclareMathOperator{\sys}{sys}

\definecolor{yuchencolor}{RGB}{0, 100, 0} 

\begin{document}
\title[Mass-capacity inequality]{Mass-capacity inequality modeled on conformally flat manifolds}

\author{Yuchen Bi}
\address[Yuchen Bi]{Mathematical Institute, Department of Pure Mathematics, University of Freiburg, Ernst-Zermelo-Stra{\ss}e 1, D-79104 Freiburg im Breisgau, Germany}
\email{yuchen.bi@math.uni-freiburg.de}

\author{Jintian Zhu}
\address[Jintian Zhu]{Institute for Theoretical Sciences, Westlake University, 600 Dunyu Road, 310030, Hangzhou, Zhejiang, People's Republic of China}
\email{zhujintian@westlake.edu.cn}

\begin{abstract}
In the spin case, we can establish a mass-capacity inequality for generalized asymptotically flat manifolds $(M,g,E)$ with nonnegative scalar curvature, where the equality implies that $(M,g)$ is harmonically conformal to $\mathbb R^n\setminus S$ for a closed bounded subset $S$ of $\mathbb R^n$ with Hausdorff dimension no greater than $\frac{n-2}{2}$.
\end{abstract}

\maketitle

\section{Introduction}

The investigation of mass constitutes a central topic in  general relativity. Notably, the basic concept of mass for isolated gravitational systems was first formulated by Arnowitt-Deser-Misner \cite{Arnowitt-Deser-Misner}, with such systems modeled as asymptotically flat manifolds. Later, Bartnik \cite{Bartnik1986} and also Chru\'sciel \cite{Chru86} showed that the mass is indeed a geometric quantity of asymptotically flat manifolds. We recall the definition of asymptotically flat manifolds and mass as follows.
\begin{definition}
    A complete Riemannian $n$-manifold $(M,g)$ with dimension $n\geq 3$ is asymptotically flat if there is a compact subset $K\subset M$ such that
\begin{itemize}
\item[(i)] the complement $M-K$ consists of finitely many ends $\{E_l\}_{l=1}^k$, where each end $E_l$ is diffeomorphic to $\mathbb R^n-\bar B_1$ with $$\bar B_1=\{x\in \mathbb R^n:|x|\leq 1\},$$
\item[(ii)]  the metric $g$ on each end $E_l$ has the expression $g=g_{ij}\mathrm dx_i\otimes\mathrm dx_j$ in the Euclidean coordinate chart, where the metric components satisfy the decay condition
\begin{equation*}\label{Eq: decay}
|g_{ij}-\delta_{ij}|+|x||\partial g_{ij}|+|x|^2|\partial^2 g_{ij}|=O\left(|x|^{-\tau}\right),
\end{equation*}
as $x\to\infty$, where $\tau$ is a positive constant greater than $\frac{n-2}{2}$,
\item[(iii)] the scalar curvature $R(g)$ belongs to $L^1(M,g)$.
\end{itemize}
For convenience, each end $E_l$ will be called an asymptotically flat end of $(M,g)$.
\end{definition}

\begin{definition}
Let $(M,g)$ be an asymptotically flat manifold and $E$ be an asymptotically flat end of $(M,g)$.
    The mass of the end $E$ is defined as
    $$m(M,g,E)=\frac{1}{2n(n-1)\omega_n}\lim_{\rho\to+\infty}\int_{S_\rho}(\partial_j g_{ij}-\partial_i g_{jj})\frac{x_i}{|x|}\,\mathrm d\sigma,$$
    where $\omega_n$ denotes the volume of the Euclidean unit ball $B_1^n$, $S_\rho$ denotes the coordinate sphere $\{|x|=\rho\}$, and $\mathrm d\sigma$ is the Euclidean area element of $S_\rho$.
\end{definition}

Many important works have revealed the source of mass. The positive mass conjecture, with the view that nonnegative energy density contributes to nonnegative mass, states that asymptotically flat manifolds with nonnegative scalar curvature have nonnegative mass on each asymptotically flat end, where the mass vanishes on one end if and only if it is the Euclidean space. This conjecture was finally proved by Schoen-Yau \cite{SY79,SY17} with the minimal surface method, and also by Witten \cite{Witten81} with the harmonic spinor method under the extra assumption that the asymptotically flat manifold is spin. Recently, the three-dimensional positive mass theorem was given new proofs by Bray-Kazaras-Khuri-Stern \cite{BKKS22} using level-set method and also by Agostiniani-Mazzieri-Oronzio \cite{AMO24} using potential method. With black holes taken into consideration, the Penrose inequality states that if an asymptotically flat manifold with only one asymptotically flat end $E$ has nonnegative scalar curvature and the boundary is an outer-minimizing horizon $\Sigma$, then we have the inequality
$$
m(M,g, E)\geq \frac{1}{2}\left(\frac{|\Sigma|_g}{|\mathbb S^{n-1}|}\right)^{\frac{n-2}{n-1}},
$$
where the equality holds if and only if $(M,g)$ is isometric to the half spatial Schwarzschild manifold
$(M_m,g_m)$ with
$$M_m=\mathbb R^n\setminus B_{m/2}\mbox{ and }g_m=\left(1+\frac{m}{2}|x|^{2-n}\right)^{\frac{4}{n-2}}g_{euc}.$$
The Penrose inequality in dimension three was first proved independently by Huisken-Ilmanen \cite{HI01} and Bray \cite{Bray01} using different flow methods. Notice that the potential approach  developed by Agostiniani-Mazzieri-Oronzio in \cite{AMO24} also works for the three-dimensional Penrose ineqality. Based on Bray's original proof, Bray-Lee \cite{BL09} removed the technical use of the Gauss-Bonnet formula  and generalized the Penrose inequality up to dimension seven. In order to prove the Penrose inequality, Bray developed a mass-capacity inequality
$$ m(M,g,E)\geq \mathfrak c(M,g,\Sigma),$$
where $\mathfrak c(M,g,\Sigma)$ denotes the capacity of $\Sigma$ in $(M,g)$. For more works related to mass-capacity inequalities for asymptotically flat manifolds, the audience can refer to the works \cite{BM08, Miao24, Miao25} and the references therein.

The study of mass is closely related to conformal geometry. Based upon the fundamental work of Trudinger \cite{Trudinger68} and Aubin \cite{Aubin76}, Schoen ultimately resolved the Yamabe problem by leveraging the positive mass theorem. As another link between conformal geometry and general relativity, Schoen-Yau \cite{Schoen-Yau} reduced the Liouville theorem in conformal geometry to a consequence of the positive mass theorem for a class of generalized asymptotically flat manifolds, which was recently established in the works \cite{Lesourd-Unger-Yau,LLU23, Zhu23,BC2005,CZ2024}.

In this paper, we continue the study of mass inequalities for generalized asymptotically flat manifolds, which are initiated from Schoen-Yau's work \cite{Schoen-Yau}. Such manifolds  are also well-known in \cite{Lesourd-Unger-Yau} as {\it asymptotically flat manifolds with arbitrary ends} with the following precise definition.
\begin{definition}
    A triple $(M,g,E)$ is said to be a generalized asymptotically flat manifold if 
    \begin{itemize}
        \item $(M,g)$ is a complete Riemannian $n$-manifold {\it without boundary}, whose dimension $n$ is no less than three;
        \item $E$ is an asymptotically flat end of $(M,g)$;
\item the scalar curvature $R(g)$ belongs to $L^1(E,g)$.
    \end{itemize}
\end{definition}

The positive mass theorem for generalized asymptotically flat manifolds was systematically researched in \cite{Lesourd-Unger-Yau,LLU23, Zhu23,BC2005,CZ2024}. Soon after, the Penrose inequality for generalized asymptotically flat manifolds was initiated by the second-named author \cite{Zhu2024} to include spatial slices of the {\it extreme} Reissner-Nordstr\"om spacetime into application, where the ultimate goal is to show the following {\it mass-systole conjecture}:
\begin{conjecture}
Let $(M,g, E)$ be a generalized asymptotically flat manifold with nonnegative scalar curvature. Then we have
$$
m(M,g, E)\geq \frac{1}{2}\left(\frac{\sys(M,g,E)}{|\mathbb S^{n-1}|}\right)^{\frac{n-2}{n-1}},
$$
where 
\begin{equation*}
\sys(M,g,E)=\inf\left\{|\Sigma|_g:\left.
\begin{array}{c}\text{$\Sigma$ is a smoothly embedded hypersurface }\\
\text{homologous to $\partial E$ in $M$}
\end{array}\right.\right\}.
\end{equation*}
Moreover, the equality holds if and only if
\begin{itemize}
    \item $(M,g)$ is isometric to the Euclidean space;
    \item or there is a strictly outer-minimizing minimal $(n-1)$-sphere $\Sigma_h$ homologous to $\partial E$ such that the region outside $\Sigma_h$ is isometric to the half spatial Schwarzschild manifold with mass $m>0$.
\end{itemize}
\end{conjecture}
The conjecture is now only known to be true when $M$ has the underlying topology $\mathbb R^3\setminus\{O\}$ concerning the work \cite{Zhu2024}. As a further exploration, we are going to consider the mass-capacity inequality for generalized asymptotically flat manifolds.

To present our theorem, we need to introduce the {\it capacity} for generalized asymptotically flat manifolds. 
Given any constant $s>1$ we shall use $E_s$ to  denote the open subset of $E$ diffeomorphic to $\mathbb R^n\setminus \bar B_s$. 
Let $\eta$ denote a fixed smooth function satisfying $\eta\equiv 0$ in $E_2$ and $\eta\equiv 1$ outside $E_1$. 
\begin{definition}\label{def:capacity}
    The capacity of a generalized asymptotically flat manifold $(M,g,E)$ is defined to be
    $$\mathfrak c(M,g,E)=\inf_{\phi\in\mathcal C}\frac{1}{n(n-2)\omega_n}\int_M|\nabla_g\phi|^2\,\mathrm d\mu_g,$$
    where  $\mathcal C$ denotes the collection of all smooth functions $\phi$ such that $\phi-\eta$ has compact support.
\end{definition}
\begin{remark}
  It is a standard fact that we can relax the test function $\phi$ to be any function in $W^{1,2}_{loc}(M,g)$ such that $\phi-\eta$ has compact support.
\end{remark}

In this paper, we are going to prove the following theorem.
\begin{theorem}\label{Thm: main1}
    Assume that $(M^n,g,E)$ is a spin generalized asymptotically flat manifold with nonnegative scalar curvature. Then we have the following mass-capacity inequality
    $$m(M,g,E)\geq 2\mathfrak c(M,g,E).$$
    If the equality holds, then $(M,g)$ is harmonically conformal  to $\mathbb R^n\setminus S$, where $S$ is a bounded closed subset of $\mathbb R^n$ with Hausdorff dimension no greater than $\frac{n-2}{2}$. In particular, we have
    $$\pi_i(M)=0\mbox{ for all }1\leq i\leq n-1- \left\lfloor\frac{n}{2}\right\rfloor.$$
\end{theorem}
\begin{remark}
In the non-spin case, the mass-capacity inequality is still true from the exactly same proof assuming the validity of the positive mass theorem. However, the rigidity seems out of reach with current techniques due to the issue of incompleteness.
\end{remark}

To include previous Bray's mass-capacity inequality as a special case, we also extend our main theorem to the setting of Riemannian manifolds with corner defined as follows.
\begin{definition}
    A triple $(M,g,\Sigma)$ will be called a Riemannian manifold with corner if $M$ is a smooth manifold equipped with a continuous metric $g$, and $\Sigma$ is a closed separating smooth hypersurface $\Sigma$ in $M$ such that the metric $g$ is actually smooth to boundary in the closure of every component of $M\setminus \Sigma$. For convenience, $\Sigma$ will be called the corner of $(M,g)$.
\end{definition}
In particular, we can consider generalized asymptotically flat manifolds with corner.
\begin{definition}
    A generalized asymptotically flat manifold $(M,g,E)$ will be called a generalized asymptotically flat manifold with corner, if there exists a closed separating smooth hypersurface $\Sigma$ such that $(M,g,\Sigma)$ is a Riemannian manifold with corner.
\end{definition}

Throughout this paper, we always take the convention that $\mathbb S^2$ has mean curvature two in $\mathbb R^3$ with respect to the outward unit normal. With extra mean-curvature condition along the corner, we can prove
\begin{theorem}\label{Thm: main2}
  Assume that $(M^n,g,E)$ is a spin generalized asymptotically flat manifold with corner, which has nonnegative scalar curvature. Assume that the sum of the mean curvatures of the corner on both sides with respect to the outward unit normal is nonnegative.  Then we have the following mass-capacity inequality
    $$m(M,g,E)\geq 2\mathfrak c(M,g,E).$$
    If the equality holds, then $(M,g)$ is smooth and harmonically conformal  to $\mathbb R^n\setminus S$, where $S$ is a bounded closed subset of $\mathbb R^n$ with Hausdorff dimension no greater than $\frac{n-2}{2}$. In particular, we have
    $$\pi_i(M)=0\mbox{ for all }1\leq i\leq n-1- \left\lfloor\frac{n}{2}\right\rfloor.$$
\end{theorem}
As a direct consequence, we can show
\begin{corollary}[Bray's mass capacity inequality]\label{Cor: Bray}
    Assume that $(M^n,g)$ is a spin asymptotically flat manifold with only one asymptotically flat end $E$, non-empty boundary $\Sigma$, and nonnegative scalar curvature.  
Assume further that the boundary $\Sigma$ is minimal. Then we have
\begin{equation}\label{Eq: Bray's mass-capacity}
    m(M,g,E)\geq \mathfrak c(M,g,\Sigma),
\end{equation}
where $\mathfrak c(M,g,\Sigma)$ denotes the capacity of $\Sigma$ in $(M,g)$ given by
$$\mathfrak c(M,g,\Sigma)=\inf_{\phi\in\mathcal C}\frac{1}{n(n-2)\omega_n}\int_M|\nabla_g\phi|^2\,\mathrm d\mu_g,$$
    where  $\mathcal C$ denotes the collection of all smooth functions $\phi$ with compact support such that $\phi\equiv 1$ on $\Sigma$.
The equality holds in \eqref{Eq: Bray's mass-capacity} if and only if $(M,g)$ is isometric to the half spatial Schwarzschild manifold.
\end{corollary}

\subsection*{Arrangements}
The remaining part of this paper is arranged as follows.
Sections \ref{Sec: inequality} and \ref{Sec: rigidity} are devoted to proving Theorem \ref{Thm: main1}. In Section \ref{Sec: inequality}, we prove the mass-capacity inequality based on the conformal method. In Section \ref{Sec: rigidity}, we establish the rigidity in the equality case of the mass-capacity inequality by constructing sufficiently many parallel spinors from an approximation procedure. In Section \ref{Sec: corner}, we show necessary modifications in the corner case and present the proofs of Theorem \ref{Thm: main2} and Corollary \ref{Cor: Bray}.

\bigskip
\section*{Acknowledgements}
The first-named author thanks Westlake University for its hospitality and support during a visit when this collaboration was initiated.

The second-named author is grateful to Professor Pengzi Miao for many helpful conversations during the BIRS-IASM workshop, {\it Recent Advances in Comparison Geometry}. He was partially supported by National Key R\&D Program of China 2023YFA1009900, NSFC grant 12401072, and the start-up fund from Westlake University.

\section{The mass-capacity inequality}\label{Sec: inequality}
In this section, we prove the desired mass-capacity inequality
$$m(M,g,E)\geq 2\mathfrak c(M,g,E)$$
for any generalized asymptotically flat manifold $(M,g,E)$ with nonnegative scalar curvature. The basic strategy is to reduce the desired mass-capacity inequality to the positive mass theorem with a use of conformal deformation.

We start with the following well-known expansion for harmonic functions on asymptotically flat ends, for which we include its proof for completeness.
\begin{lemma}\label{Lem: expansion}
    If $v$ is a harmonic function on some $E_s$ such that $v(x)\to 0$ as $x\to \infty$, then we have the expansion
    $$v=c|x|^{2-n}+w,$$
    where $c$ is a real constant and the error term $w$ satisfies $$w=O_2(|x|^{2-n-\tau})\mbox{ as } x\to\infty,$$ which means
    \begin{equation*}
    |w|+|x||\partial w|+|x|^2|\partial^2 w|=O(|x|^{2-n-\tau})\mbox{ as }x\to \infty.
    \end{equation*}    
\end{lemma}
\begin{proof}
    Without loss of generality, we can extend $(E_s,g)$ to an asymptotically flat manifold $(\mathbb R^n,\hat g)$ and the function $v$ to a function $\hat v$ defined on $\mathbb R^n$. Note that $\hat v$ is a harmonic function on $(\mathbb R^n,\hat g)$ outside a compact subset. Let $W^{k,p}_\delta$ be the weighted Sobolev space defined as the completion of $C_c^\infty$ with respect to the following norm
    $$\|u\|_{W^{k,p}_\delta}=\sum_{|j|=0}^k\left(\int_{\mathbb R^n}|D^ju|^p\rho^{-(\delta-|j|)p-n}\,\mathrm dx\right)^{\frac{1}{p}},$$
    where $\rho$ denotes the modified radial function given by
    $$\rho=\max\{1,|x|\}.$$
    Fix two constants $\delta\in (2-n,0)$ and $p>n$. It follows from \cite[Proposition 2.2]{Bartnik1986} as well as the maximum principle that $\hat v$ is a function in $W^{2,p}_\delta$. Then by \cite[Theorem 1.17]{Bartnik1986} we have the following expansion
    $$v=c|x|^{2-n}+O(|x|^{2-n-\tau})\mbox{ as }x\to \infty.$$
    Denote $w=v-c|x|^{2-n}$. Then it follows from the classical Schauder estimate \cite[Theorem 6.2]{Gilbarg-Trudinger} that $w=O_2(|x|^{2-n-\tau})$ as $x\to\infty$.
\end{proof}

\begin{corollary}\label{Cor: function u}
    There is a positive harmonic function $u$ on $(M,g)$ such that we have the expansion 
    $$u(x)=1-\mathfrak c(M,g,E)\cdot|x|^{2-n}+w,$$
    with $w=O_2(|x|^{2-n-\tau})$ as $x\to\infty$.
\end{corollary}
\begin{proof}
    Let $\{V_i\}_{i=1}^\infty$ be a smooth open exhaustion of $M$ such that each $V_i$ contains $E$ and that $V_i\setminus E$ is bounded for all $i$. We may assume $\{V_i\}_{i=1}^\infty$ to be increasing without loss of generality. Let us define
    \begin{equation}\label{Eq: ci}
    \mathfrak c_i=\inf_{\phi\in\mathcal C_i}\frac{1}{n(n-2)\omega_n}\int_M|\nabla_g\phi|^2\,\mathrm d\mu_g,
    \end{equation}
     where $\mathcal C_i$ denotes the collection of all smooth functions $\phi$ such that $\phi=1$ on $\partial V_i$ and $\phi\equiv 0$ around the infinity of $E$. 
     Through a standard minimizing procedure, we can find a harmonic function $v_i:\bar V_i\to [0,1]$ such that $v_i=1$ on $\partial V_i$, $v_i(x)\to 0$ as $x\to\infty$ and
     $$\int_{V_i}|\nabla v_i|^2\,\mathrm d\mu_g=n(n-2)\omega_n \mathfrak c_i.$$
     It follows from Lemma \ref{Lem: expansion} that $v_i$ has the expansion
     $$v_i=c_i|x|^{2-n}+w_i,$$
     where $c_i$ is a constant and $w_i=O_2(|x|^{2-n-\tau})$ as $x\to\infty$. From integration by parts we can derive $c_i=\mathfrak c_i$ and also
     \begin{equation}\label{Eq: local capacity}
         \int_{\partial E}\frac{\partial v_i}{\partial \vec n}\,\mathrm d\sigma_g=n(n-2)\omega_n\mathfrak c_i,
     \end{equation}
where $\vec n$ denotes the outward unit normal of $\partial E$ with respect to $E$.

Notice that $v_i$ is decreasing as $i$ increases. In particular, the functions $v_i$ converge to a harmonic function $v$ satisfying $0\leq v\leq 1$ and $v(x)\to 0$ as $x\to\infty$. From the Schauder estimates the convergence is smooth in every bounded subset, so we know from \eqref{Eq: local capacity} that the constants $\mathfrak c_i$ converge to some constant $\mathfrak c$ and
$$\int_{\partial E}\frac{\partial v}{\partial \vec n}\,\mathrm d\sigma_g=n(n-2)\omega_n\mathfrak c.$$
From Lemma \ref{Lem: expansion} and integration by parts we know that $v$ has the expansion
$$v=\mathfrak c|x|^{2-n}+w\mbox{ with }w=O_2(|x|^{2-n-\tau})\mbox{ as }x\to\infty.$$

Let us verify $\mathfrak c=\mathfrak c(M,g,E)$. By definition, we have $\mathfrak c_i\geq \mathfrak c(M,g,E)$ and so $\mathfrak c\geq \mathfrak c(M,g,E)$. For the opposite direction, we take a sequence of smooth functions $\phi_j$ in the class $\mathcal C$ such that
$$n(n-2)\omega_n\int_M|\nabla \phi_j|^2\,\mathrm d\mu_g\to \mathfrak c(M,g,E)\mbox{ as }j\to \infty.$$
For each $\phi_j$ we can find some $V_i$ containing the support of $\phi_j-\eta$ and so $\phi_j$ is in the class $\mathcal C_i$. Therefore, we have
$$n(n-2)\omega_n\int_M|\nabla \phi_j|^2\,\mathrm d\mu_g\geq \mathfrak c_i.$$
After taking limit, we obtain $\mathfrak c(M,g,E)\geq \mathfrak c$.

Now we complete the proof by taking $u=1-v$. The positivity of $u$ comes from the strong maximum principle and the fact $u(x)\to 1$ as $x\to\infty$.
\end{proof}

\begin{proposition}\label{Prop: mass capacity}
    If $(M,g)$ has nonnegative scalar curvature, then we have
    $$m(M,g,E)\geq 2 \mathfrak c(M,g,E).$$
\end{proposition}
\begin{proof}
    Given any constant $\varepsilon>0$ we define
    \begin{equation}\label{Eq: approximation metric}
    g_\varepsilon=u_\varepsilon^{\frac{4}{n-2}}g\mbox{ with } u_\varepsilon=\frac{u+\varepsilon}{1+\varepsilon},
    \end{equation}
    where $u$ is the positive harmonic function from Corollary \ref{Cor: function u}.
    Note that $u_\varepsilon$ has a positive lower bound, so $(M,g_\varepsilon)$ remains to be a complete Riemannian manifold without boundary. It follows from the expansion of $u$ in Corollary \ref{Cor: function u} that $(M,g_\varepsilon,E)$ remains to be a generalized asymptotically flat manifold and that we have
$$m(M,g_\varepsilon,E)=m(M,g,E)-2(1+\varepsilon)^{-1}\mathfrak c(M,g,E).$$
    Clearly, $u_\varepsilon$ is a harmonic function on $(M,g)$ since $u$ is. We see that $(M,g_\varepsilon)$ has nonnegative scalar curvature. Then it follows from the spin positive mass theorem (see \cite{BC2005, CZ2024} for instance) for generalized asymptotically flat manifolds that we have
    $$m(M,g_\varepsilon,E)\geq 0,$$ which yields
    $$m(M,g,E)\geq 2(1+\varepsilon)^{-1}\mathfrak c(M,g,E).$$
    The proof is now completed by letting $\varepsilon\to 0$.
\end{proof}
\begin{remark}
    If one simply takes the function $u$ to be the conformal factor, then the positive mass theorem may not be applied since the Riemannian manifold after conformal deformation can be incomplete. This explains the reason why we have to use an approximation argument for the mass-capacity inequality.
\end{remark}

\section{Rigidity}\label{Sec: rigidity}

In this section, we deal with the rigidity of the mass-capacity inequality and always assume
$$m(M,g,E)=2\mathfrak c(M,g,E).$$
To prove rigidity, we have to deal with the {\it possibly incomplete} Riemannian manifold $(M,\bar g,E)$ with
$$\bar g=u^{\frac{4}{n-2}}g,$$
where $u$ is the positive harmonic function coming from Corollary \ref{Cor: function u}.
 
First let us show the flatness of $(M,\bar g)$ based on the spinor method from \cite{CZ2024}. The basic idea is to construct sufficiently many parallel harmonic spinors on $(M,\bar g)$.

\subsection{A brief review on spin geometry}
Since $M$ is spin, we can take a spin structure $P_{spin}\to P_{SO}(M,\bar g)$. Let 
$$\rho:spin(n)\to End(\Delta_n)$$ denote the complex spin representation. Correspondingly, the complex spinor bundle $\mathcal S\to M$ is defined to be
$$\mathcal S=P_{spin}\times_\rho \Delta_n.$$ 
In particular, we have the Clifford multiplication $$\bar c:Cl(T_xM,\bar g_x)\to End(\mathcal S_x).$$ 
Denote the Levi-Civita connection with respect to the metric $\bar g$ by $\overline\nabla$. Recall that this induces a spinor connection, still denoted by $\overline\nabla$ for simplicity, on the spinor bundle $\mathcal S$,  which satisfies the Leibniz rule
$$\overline\nabla_X (\bar c(\varphi)s)=\bar c(\overline\nabla_X \varphi)s+\bar c(\varphi)\overline\nabla_X s$$
for any smooth sections
$$X\in \Gamma(TM),\,\varphi \in \Gamma(Cl(M,\bar g))\mbox{ and }s\in \Gamma(\mathcal S).$$
Moreover, there is a Hermitian metric $\langle\cdot,\cdot\rangle$ on the spinor bundle $\mathcal S$ such that
\begin{itemize}
\item[(i)] we have
$$\langle \bar c(v)s_1,\bar c(v)s_2\rangle=\langle s_1,s_2\rangle$$
for any $\bar g$-unit vector $v\in T_xM$ and $s_i\in \mathcal S_x$ at any point $x\in M$,
\item[(ii)] and we have
$$ X\langle s_1,s_2\rangle=\langle \overline\nabla_X s_1,s_2\rangle+\langle s_1,\overline\nabla_X s_2\rangle$$
for any smooth sections $X\in \Gamma(TM)$ and $s_i\in \Gamma(\mathcal S)$.
\end{itemize}
Based on the Clifford multiplication and the spinor connection, the Dirac operator $\overline{\mathcal D}$ on the spinor bundle $\mathcal S$ is defined to be 
$$\overline{\mathcal D}=\bar c(\bar e_i)\overline\nabla_{\bar e_i},$$
where $\{\bar e_i\}$ is a $\bar g$-orthonormal basis of $T_xM$.

Suppose that $\tilde g$ is another metric on $M$ which is conformal to the metric $\bar g$. Then the orthonormal frame bundle $P_{SO}(M,\tilde g)$ has a natural isomorphism to $P_{SO}(M,\bar g)$ up to scaling of frames. Therefore, we can take the same spin structure $$P_{spin}\to P_{SO}(M,\tilde g)\cong P_{SO}(M,\bar g)$$ and also the same spinor bundle $\mathcal S$ as before. Denote the associated Clifford multiplication and the spinor connection by $\tilde c$ and $\widetilde \nabla$ respectively. Then the corresponding Dirac operator $\widetilde{\mathcal D}$ is defined by
$$\widetilde{\mathcal D}=\tilde c(\tilde e_i)\widetilde \nabla_{\tilde e_i},$$
where $\{\tilde e_i\}$ is a $\tilde g$-orthonormal basis of $T_xM$.
Since $\tilde g$ is conformal to $g$, we can write 
$$\tilde g=e^{2w}\bar g$$
for some smooth function $w$ on $M$. Through a straightforward computation from Appendix \ref{App: A}, we have the following relations
$$\tilde c(v)=e^w\bar c(v)\mbox{ for any }v\in T_xM,$$
and
\begin{equation}\label{Eq: connection relation}
    \widetilde \nabla_{(\cdot)}=\overline \nabla_{(\cdot)}+\frac{1}{2}\bar c(\overline \nabla w)\bar c(\cdot)+\frac{1}{2}\bar g(\overline \nabla w,\cdot).
\end{equation}
In particular, we still have
\begin{equation}\label{Eq: compatible}
     X\langle s_1,s_2\rangle=\langle \widetilde\nabla_X s_1,s_2\rangle+\langle s_1,\widetilde\nabla_X s_2\rangle
\end{equation}
for any smooth sections $X\in \Gamma(TM)$ and $s_i\in \Gamma(\mathcal S)$.

In our application, we consider those approximation metrics $g_\varepsilon$ from \eqref{Eq: approximation metric} conformal to the metric $\bar g$. With the construction above, we can define the Clifford multiplication, the spinor connection and also the Dirac operator on $\mathcal S$ correspondingly, denoted by $\tilde c_\varepsilon$, $\widetilde \nabla^\varepsilon$ and $\widetilde{\mathcal D}^\varepsilon$ respectively. 

\subsection{The method of Callias operator}
We follow the set-up in the work \cite{CZ2024} by Cecchini-Zeidler. Let us consider the double spinor bundle
$$\hat {\mathcal S}=\mathcal S\oplus \mathcal S.$$
It is clear that there is a natural Hermitian metric $\widehat{\langle\cdot,\cdot\rangle}$ on $\hat {\mathcal S}$ induced from that on $\mathcal S$. Similarly, there is a natural spinor connection on $\hat{\mathcal S}$ induced from that on $\mathcal S$, denoted by $\widehat\nabla^\varepsilon$. Moreover, we can introduce the twisted Clifford multiplication 
$$\hat c_\varepsilon:Cl (T_xM,g_{\varepsilon})\to End(\hat S_x)\mbox{ with }\hat c_\varepsilon(\cdot):=\left(
\begin{array}{cc}
0& \tilde c_\varepsilon(\cdot)\\
\tilde c_\varepsilon(\cdot)&0
\end{array}\right)$$
and the twisted Dirac operator $\widehat{\mathcal D}^\varepsilon$ by
\[\widehat{\mathcal D}^\varepsilon=\left(
\begin{array}{cc}
0&\widetilde{\mathcal D}^\varepsilon\\
\widetilde{\mathcal D}^\varepsilon&0
\end{array}\right).\]
Given any smooth function $\psi$ on $M$ we can also define the {\it Callias operator} $\widehat B_\psi^\varepsilon$ by
$$\widehat B_\psi^\varepsilon=\widehat{\mathcal D}^\varepsilon+\psi \sigma,$$
 where $\sigma$ is the involution on $\hat{\mathcal S}$ given by
 $$\sigma=\left(
\begin{array}{cc}
0& -\mathfrak i\\
\mathfrak i&0
\end{array}\right)\mbox{ where }\mathfrak i=\sqrt{-1}.$$

Let $U$ be a smooth open subset of $M$ such that $U\Delta E$ has compact closure, where $E$ denotes the distinguished end of the generalized asymptotically flat manifold $(M,g,E)$. Assume that $U$ has non-empty boundary $\partial U$. We use $\nu_\varepsilon$ to denote the inward unit normal of $\partial U$ in $U$ with respect to the metric $g_\varepsilon$. Then the {\it chirality operator} is defined to be
$$\widehat\chi_\varepsilon=\hat c_\varepsilon(\nu_\varepsilon) \sigma: \hat {\mathcal S}|_{\partial U}\to \hat {\mathcal S}|_{\partial U}.$$

Now, let us introduce several function spaces for analysis. We use $$L^p_{loc}(\overline U,\hat{\mathcal S})$$ to denote the collection of all spinors which are $L^p$ in all compact subsets of $\overline U$. We point out that the definition of the space $L^p_{loc}(\overline U,\hat{\mathcal S})$ does not depend on the choice of the reference metric, which we just omit in above notation. 

In a similar way, we use 
$$W^{k,p}_{loc}(\overline U,g_\varepsilon,\hat{\mathcal S})$$ to denote the collection of all spinors which are $W^{k,p}$ in all compact subsets of $\overline U$, where the derivative and the integration are both taken with respect to the metric $g_\varepsilon$. To further measure the decay rate of functions at infinity, we introduce the weighted Lebesgue spaces and weighted Sobolev spaces.
For any $\delta\in\mathbb R$, the weighted Lebesgue space $$L^p_{\delta}(U,g_\varepsilon,\hat{\mathcal S})$$ is defined to be the collection of all spinors in $L^p_{loc}(\overline U,\hat{\mathcal S})$ with finite norm
$$\|\hat s\|_{L^p_{\delta}(U,g_\varepsilon,\hat{\mathcal S})}=\left(\int_U\widehat{|\hat s|}^{p}\rho^{-\delta p-n}\,\mathrm d\mu_{g_\varepsilon}\right)^{\frac{1}{p}},$$
where 
$$\widehat{|\hat s|}=\widehat{\langle \hat s,\hat s\rangle}^{\frac{1}{2}}$$ and $\rho$ is the function given by 
\[
\rho=\left\{\begin{array}{cl}
|x|&\hspace{-2mm}\mbox{in }E,\\
1&\hspace{-2mm}\mbox{outside }E.\end{array}\right.
\]
Similarly, the weighted Sobolev space 
$$W^{k,p}_\delta(U,g_\varepsilon,\hat{\mathcal S})$$ is defined to be the collection of all spinors in $W^{k,p}_{loc}(\overline U,g_\varepsilon,\hat{\mathcal S})$ with finite norm
$$\|\hat s\|_{W^{k,p}_\delta(U,g_\varepsilon,\hat{\mathcal S})}:=\sum_{i=0}^k\|\widehat \nabla^\varepsilon\hat s\|_{L^p_{\delta-i}(U,g_\varepsilon,\hat{\mathcal S})}.$$
To take the boundary condition into consideration,
we will use  $C^\infty(U,\hat{\mathcal S};\widehat \chi_\varepsilon)$ to denote the collection of smooth spinors $\hat s\in C^\infty(U,\hat{\mathcal S})$ satisfying $$\widehat\chi_\varepsilon\hat s|_{\partial U}=\hat s|_{\partial U}.$$
Similar notations will be used in the case of weighted Sobolev spaces, where the restriction of a spinor on the boundary $\partial U$ is always understood in the sense of trace.

As a preparation, let us review some preliminary results from \cite{CZ2024}. Denote
$$q^*=\frac{n-2}{2}.$$

\begin{lemma}\label{Lem: nonnegative index}
    Given any function $\psi\in C_c^\infty(\overline U,\mathbb R)$, the Callias operator $$\widehat { B}^\varepsilon_\psi:H^1_{-q^*}(U,g_\varepsilon,\hat{\mathcal S};\widehat\chi_\varepsilon)\to L^2(U,g_\varepsilon,\hat{\mathcal S})$$
    is a Fredholm operator with nonnegative index.
\end{lemma}
\begin{proof}
    This is \cite[Theorem 2.12]{CZ2024}.
\end{proof}

Let $\{\bar e_i\}$ be the $\bar g$-orthonormal frame on $E$ coming from orthonormalizing the coordinate frame $\{\partial_{x_i}\}$ with respect to the metric $\bar g$. Then the frame $\{\bar e_i\}$ determines a smooth section of the frame bundle $P_{SO}(M,\bar g)$, which can be further lifted to a smooth section of the spin structure $P_{spin}$. This induces an isomorphism 
\begin{equation}\label{Eq: isomorphism}
    \hat {\mathcal S}|_E\cong E\times (\Delta_n\oplus\Delta_n).
\end{equation}

\begin{definition}
    A smooth spinor 
    $$\hat s_0\in C^\infty(E,\hat{\mathcal S}|_E)$$ is called constant if it is a constant section of $E\times(\Delta_n\oplus\Delta_n)$ through the isomorphism \eqref{Eq: isomorphism}. Moreover, a spinor 
    $$\hat s\in H^1_{loc}(\overline U,g_\varepsilon,\hat{\mathcal S})$$ is said to be {\it asymptotically constant with respect to $g_\varepsilon$} if we have
$$\hat s|_E-\hat s_0\in H^1_{-q^*}(E,g_\varepsilon,\hat{\mathcal S})$$
for some constant spinor $\hat s_0\in C^\infty(E,\hat{\mathcal S}|_E)$. In particular, we can define $$\|\hat s\|_{\infty}:=\widehat{|\hat s_0|}\in [0,+\infty).$$
\end{definition}

\begin{lemma}\label{Lem: mass control energy}
    Given a smooth function $\psi\in C_c^\infty(\overline U,\mathbb R)$ and a spinor $\hat s\in H^1_{loc}(\overline U,g_\varepsilon,\hat{\mathcal S};\widehat\chi_\varepsilon)$ which is asymptotically constant with respect to $g_\varepsilon$, we have
    \[
    \begin{split}
        \frac{n-1}{2}&\omega_{n-1}\cdot m(M,g_\varepsilon,E)\cdot\|\hat s\|_{\infty}^2+\|\widehat B^\varepsilon_\psi\hat s\|^2_{L^2(U,g_\varepsilon,\hat{\mathcal S})}\\
        &\geq  \|\widehat\nabla^\varepsilon\hat s\|^2_{L^2(U,g_\varepsilon,\hat{\mathcal S})}+\int_U\theta_\psi\widehat{|\hat s|}^2\mathrm d\mu_{g_\varepsilon}+\int_{\partial U}\eta_\psi \widehat{|\hat s|}^2\mathrm d\sigma_{g_\varepsilon},
    \end{split}
    \]
    where
    $$\theta_\psi=\frac{R(g_\varepsilon)}{4}+\psi^2-|\mathrm d\psi|_{g_\varepsilon}$$
    and 
    $$\eta_\psi=\frac{H(g_\varepsilon)}{2}+\psi|_{\partial U},$$
    where $H(g_\varepsilon)$ denotes the mean curvature of $\partial U$ with respect to the metric $g_\varepsilon$ and outward unit normal.
\end{lemma}
\begin{proof}
This is \cite[Proposition 2.5]{CZ2024}. We include the proof for the convenience of the audience. By definition it is direct to verify
    $$\widehat{\mathcal D}^\varepsilon\sigma=-\sigma \widehat{\mathcal D}^\varepsilon\mbox{ and }\widehat D^\varepsilon \psi=\hat c_\varepsilon(\widetilde\nabla^\varepsilon\psi)+\psi\widehat{\mathcal D}^\varepsilon.$$
    Combined with the Schr\"odinger-Lichnerowicz formula we can compute
    \[
    \begin{split}
        (\widehat B^\varepsilon_\psi)^2&=(\widehat{\mathcal D}^\varepsilon)^2+\widehat D^\varepsilon \psi\sigma+\psi\sigma\widehat{\mathcal D}^\varepsilon+\psi^2\\
        &=\widehat\nabla^{\varepsilon,*}\widehat\nabla^\varepsilon+\frac{R(g_\varepsilon)}{4}+\hat c_\varepsilon(\widetilde\nabla^\varepsilon\psi)\sigma+\psi^2.
    \end{split}
    \]
    Acting both sides on $\hat s$, pairing with $\hat s$, and integrating over $(U,g_\varepsilon)$, we obtain
    \begin{equation}\label{Eq: 01}
    \begin{split}
        &\|\widehat B^\varepsilon_\psi\hat s\|^2_{L^2(U,g_\varepsilon,\hat{\mathcal S})}+\int_{\partial U\cup S_\infty}\widehat{\langle \hat B^\varepsilon_\psi\hat s,\hat c_\varepsilon(\nu_\varepsilon)\hat s\rangle}\,\mathrm d\sigma_{g_\varepsilon}\\
        &=\|\widehat\nabla^\varepsilon\hat s\|^2_{L^2(U,g_\varepsilon,\hat{\mathcal S})}+\int_U\left(\frac{R(g_\varepsilon)}{4}+\psi^2\right)\widehat{|\hat s|}^2+\hat c_\varepsilon(\widetilde\nabla^\varepsilon\psi)\widehat{\langle\sigma\hat s,\hat s\rangle }\,\mathrm d\mu_{g_\varepsilon}\\
        &\qquad\qquad+\int_{\partial U\cap S_\infty}\widehat{\langle \widehat\nabla^\varepsilon_{\nu_\varepsilon}\hat s,\hat s\rangle}\,\mathrm d\sigma_{g_\varepsilon},
    \end{split}
    \end{equation}
    where the integral over $S_\infty$ is understood to be the limit of the integral over coordinate spheres $\{|x|=r\}$ as $r\to +\infty$, and $\nu_\varepsilon$ is the inward unit normal of $\partial U$ in $U$ as well as the unit normal of $S_\infty$ opposite to the infinity with respect to the metric $g_\varepsilon$. Denote
    $$I=\widehat{\langle \widehat\nabla^\varepsilon_{\nu_\varepsilon}\hat s,\hat s\rangle}-\widehat{\langle \hat B^\varepsilon_\psi\hat s,\hat c_\varepsilon(\nu_\varepsilon)\hat s\rangle}$$
    and write $\hat s=(s_1,s_2)$. Then we can compute
    $$I=\sum_{i=1}^2\langle \widetilde \nabla^\varepsilon_{\nu_\varepsilon}s_i+ \tilde c_\varepsilon(\nu_\varepsilon)\widetilde{\mathcal D}^\varepsilon s_i,s_i\rangle\mbox{ on }S_\infty$$
    and it follows from \cite[Corollary 5.15]{DanLeeGeometricRelativity} that we have
    \begin{equation}\label{Eq: 02}
        \int_{S_\infty}I\,\mathrm d\sigma_{g_\varepsilon}=\frac{n-1}{2}\omega_{n-1}\cdot m(M,g_\varepsilon,E)\cdot\|\hat s\|^2_{\infty}.
    \end{equation}
    Define
    $$\widehat {\mathcal D}^\varepsilon_\partial=\hat c_\varepsilon(\nu)\hat c_\varepsilon(e_{i,\partial})\widehat\nabla^\varepsilon_{e_{i,\partial}}-\frac{1}{2}H(g_\varepsilon),$$
    where $\{e_{i,\partial}\}$ denotes an orthonormal frame on $\partial U$ with respect to the metric $g_\varepsilon$. It is direct to verify
    $$\widehat{\mathcal D}^\varepsilon_\partial\widehat\chi_\varepsilon+\widehat\chi_\varepsilon\widehat{\mathcal D}^\varepsilon_\partial=0,\,\widehat{\langle\widehat\chi_\varepsilon\hat s_1,\hat s_2\rangle}=\widehat{\langle\hat s_1,\widehat\chi_\varepsilon\hat s_2\rangle},$$
    and
    $$\int_{\partial U}\widehat{\langle\widehat{\mathcal D}^\varepsilon_\partial\hat s_1,\hat s_2\rangle}\,\mathrm d\sigma_{g_\varepsilon}=\int_{\partial U}\widehat{\langle\hat s_1,\widehat{\mathcal D}^\varepsilon_\partial\hat s_2\rangle}\,\mathrm d\sigma_{g_\varepsilon}.$$
    Recall that we have $\widehat\chi_\varepsilon\hat s=\hat s$ on $\partial U$. Then we can derive
    \begin{equation}\label{Eq: 03}
    \begin{split}
        \int_{\partial U}I\,\mathrm d\sigma_{g_\varepsilon}&=\int_{\partial U}\widehat{\langle \widehat{\mathcal D}^\varepsilon_\partial\hat s,\hat s\rangle}+\frac{1}{2}H(g_\varepsilon)\widehat{|\hat s|}^2+\psi\widehat{\langle \widehat\chi_\varepsilon\hat s,\hat s\rangle}\,\mathrm d\sigma_{g_\varepsilon}\\
        &=\frac{1}{2}\int_{\partial U}\widehat{\langle \widehat{\mathcal D}^\varepsilon_\partial\hat s,\widehat\chi_\varepsilon\hat s\rangle}+\widehat{\langle \widehat\chi_\varepsilon\hat s,\widehat{\mathcal D}^\varepsilon_\partial\hat s\rangle}\,\mathrm d\mu_{g_\varepsilon}+\int_{\partial U} \eta_\psi\widehat{|\hat s|}^2\mathrm d\sigma_{g_\varepsilon}\\
        &=\int_{\partial U} \eta_\psi\widehat{|\hat s|}^2\mathrm d\sigma_{g_\varepsilon},
    \end{split}
     \end{equation}
    where we have used
    $$\int_{\partial U}\widehat{\langle \widehat{\mathcal D}^\varepsilon_\partial\hat s,\widehat\chi_\varepsilon\hat s\rangle}+\widehat{\langle \widehat\chi_\varepsilon\hat s,\widehat{\mathcal D}^\varepsilon_\partial\hat s\rangle}\,\mathrm d\sigma_{g_\varepsilon}=\int_{\partial U}\widehat{\langle (\widehat\chi_\varepsilon\widehat{\mathcal D}^\varepsilon_\partial+\widehat\chi_\varepsilon\widehat{\mathcal D}^\varepsilon_\partial)\hat s,\hat s\rangle}\mathrm d\sigma_{g_\varepsilon}=0$$
    in the final line. On the other hand, we have
    \begin{equation}\label{Eq: 04}
        \int_U\hat c_\varepsilon(\widetilde\nabla^\varepsilon\psi)\widehat{\langle\sigma\hat s,\hat s\rangle }\,\mathrm d\sigma_{g_\varepsilon}\geq -\int_U|\mathrm d\psi|_{g_\varepsilon}\widehat{|\hat s|}^2\mathrm d\sigma_{g_\varepsilon}.
    \end{equation}
We complete the proof by combining \eqref{Eq: 01}-\eqref{Eq: 04}.
\end{proof}

\subsection{The flatness} \label{subsection: flatness}
In this subsection, we show the flatness of $(M,\bar g)$. If $M$ has only one end, then $(M,\bar g)$ is a spin asymptotically flat manifold in the classical sense, which has nonnegative scalar curvature and zero mass, where the flatness of $(M,\bar g)$ follows from the spin positive mass theorem \cite{BC2005,CZ2024} for generalized asymptotically flat manifolds. Therefore, in the remaining discussion we will always assume that $M\setminus E$ is unbounded, and consequently $(M,\bar g)$ could be incomplete.

First let us construct approximation harmonic spinors.

\begin{lemma}\label{Lem: approximation spinor}
    Given any constant spinor $\hat s_0\in C^\infty(E,\hat{\mathcal S}|_E)$, there is a spinor $\hat s\in H^1_{loc}(M,g_\varepsilon,\hat{\mathcal S})$ asymptotically constant to $\hat s_0$ with respect to $g_\varepsilon$ such that we have $\widehat{\mathcal D}^\varepsilon \hat s=0$ in $L^2(M,g_\varepsilon,\hat{\mathcal S})$ and
    \begin{equation}\label{Eq: energy estimate}  \|\widehat\nabla^\varepsilon\hat s\|^2_{L^2(M,g_\varepsilon,\hat{\mathcal S})}\leq \frac{n-1}{2}\omega_{n-1}\cdot m(M,g_\varepsilon,E)\cdot\|\hat s_0\|_\infty^2.
    \end{equation}
\end{lemma}
\begin{proof}
This will be done through an exhaustion procedure. It is standard that we can take an increasing smooth open exhaustion $\{U_k\}_{k=1}^\infty$ of $M$ such that all $U_k\setminus E$ are bounded. Recall that $M\setminus E$ is unbounded from our assumption, so each $U_k$ has non-empty smooth boundary $\partial U_k$. Denote
    $$\mathcal A=\{1\leq |x|\leq 2\}\subset E.$$
   Then it is easy to construct smooth functions $\psi_k$ on $U_k$ such that
    \begin{itemize}
        \item $\psi_k\equiv 0$ in $E\setminus \mathcal A$ and $2\psi_k+H_{\partial U_k}(g_\varepsilon)\geq 0$ on $\partial U_k$;
        \item $F_k:=\psi_k^2-|\mathrm d\psi_k|_{g_\varepsilon}$ are nonnegative outside $\mathcal A$ and the functions $\psi_k$ converge smoothly to the zero function in $\mathcal A$ as $k\to+\infty$.
    \end{itemize}

We claim that the map $$\widehat B^\varepsilon_{\psi_k}: H^1_{-q^*}(U_k,g_\varepsilon,\hat{\mathcal S};\widehat\chi_\varepsilon)\to L^2(U_k,g_\varepsilon,\hat{\mathcal S})$$
is surjective for sufficiently large $k$. According to Lemma \ref{Lem: nonnegative index} it suffices to show $\ker\widehat B^\varepsilon_{\psi_k}=0$. To show this we take any spinor 
$\hat s\in \ker\widehat B^\varepsilon_{\psi_k}$, then from Lemma \ref{Lem: mass control energy} and the construction of $\psi_k$ we have
$$\int_{\mathcal A}|F_k|\widehat{|\hat s|}^2\mathrm d\mu_{g_\varepsilon}\geq -\int_{U_k}\theta_{\psi_k}\widehat{|\hat s|}^2\mathrm d\mu_{g_\varepsilon}\geq \|\widehat\nabla^\varepsilon\hat s\|^2_{L^2(U_k,g_\varepsilon,\hat{\mathcal S})}.$$
On the other hand, since the weighted Sobolev inequality is true due to the asymptotical flatness of $g_\varepsilon$ and the Kato inequality holds due to \eqref{Eq: compatible}, with a use of H\"older inequality there is a positive constant $C$ independent of $k$ such that 
$$\|\widehat\nabla^\varepsilon\hat s\|^2_{L^2(E,g_\varepsilon,\hat{\mathcal S})}\geq \|\widehat\nabla^\varepsilon\widehat{|\hat s|}\|^2_{L^2(E,g_\varepsilon)}\geq C\int_{\mathcal A}\widehat{|\hat s|}^2\mathrm d\mu_{g_\varepsilon}.$$
Since we can guarantee $\sup_{\mathcal A}|F_k|<C$ for sufficiently large $k$, we obtain $\hat s\equiv 0$ in $\mathcal A$, and then the unique continuation property yields $\hat s\equiv 0$ in the whole $U_k$.

Now let us fix a smooth cut-off function $\zeta$ on $M$ such that $\zeta\equiv 1$ around the infinity of $E$ and $\zeta \equiv 0$ outside $E$. From the previous discussion, since we have $\widehat B^\varepsilon_{\psi_k}(\zeta \hat s_0)\in L^2(U_k,g_\varepsilon,\hat{\mathcal S})$ from a direct computation, we can find a spinor $$\hat s_k^* \in H^1_{-q^*}(U_k,g_\varepsilon,\hat{\mathcal S};\widehat\chi_\varepsilon)$$
satisfying
$$\widehat B^\varepsilon_{\psi_k}\hat s_k^*=-\widehat B^\varepsilon_{\psi_k}(\zeta \hat s_0).$$
From a similar analysis as before based on Lemma \ref{Lem: mass control energy} and the construction of $\psi_k$ we can derive
$$\|\widehat \nabla^\varepsilon\hat s_k^*\|_{L^2(U_k,g_\varepsilon,\hat{\mathcal S})}^2\leq C\|\widehat B^\varepsilon_{\psi_k}(\zeta\hat s_0)\|_{L^2(E,g_\varepsilon,\hat{\mathcal S})}^2\leq C\|\zeta\hat s_0\|_{H^1_{-q^*}(E,g_\varepsilon,\hat{\mathcal S})}^2$$
for all $k$ large enough, where $C$ is a constant independent of $k$. In particular, we know that the spinors $\hat s_k^*$ converge to a spinor $\hat s_\infty^*$ weakly in $H^1_{loc}(M,g_\varepsilon,\hat{\mathcal S})$ up to a subsequence. Since the norm is lower-semicontinuous with respect to the weak convergence, we have 
$$\hat s_\infty^*\in H^1_{-q^*}(E,g_\varepsilon,\hat{\mathcal S}),$$ and so the spinor 
$$\hat s:=\hat s_\infty^*+\zeta\hat s_0$$ is asymptotically constant to $\hat s_0$ with respect to $g_\varepsilon$. The weak convergence yields that $\hat s$ is a weak solution of $\widehat {\mathcal D}^\varepsilon \hat s=0$. Combined with the fact that $\hat s$ has its derivative in $L^2(M,g_\varepsilon,\hat{\mathcal S})$, we obtain $\widehat {\mathcal D}^\varepsilon\hat s= 0$ in $L^2(M,g_\varepsilon,\hat{\mathcal S})$.

Denote $$s_k:=\hat s_k^*+\zeta \hat s_0.$$
Note that the derivatives $\widehat\nabla^\varepsilon \hat s_k$ converge to $\widehat\nabla^\varepsilon\hat s$ weakly in $L^2(M,g_\varepsilon,\hat{\mathcal S})$, and that from Lemma \ref{Lem: mass control energy} we have
$$
\|\widehat\nabla^\varepsilon\hat s_k\|^2_{L^2(M,g_\varepsilon,\hat{\mathcal S})}\leq \frac{n-1}{2}\omega_{n-1}\cdot m(M,g_\varepsilon,E)\cdot\|\hat s_0\|_\infty^2+\int_{\mathcal A}|F_k|\widehat{|\hat s_k|}^2\mathrm d\mu_{g_\varepsilon}.$$
Using the lower-semicontinuity of $L^2$-norm under weak convergence and the estimate
$$\int_{\mathcal A}|F_k|\widehat{|\hat s_k|}^2\mathrm d\mu_{g_\varepsilon}\leq o(1)\cdot\|\widehat\nabla^\varepsilon\hat s_k\|^2_{L^2(M,g_\varepsilon,\hat{\mathcal S})}\mbox{ as }k\to+\infty,$$
we finally arrive at the desired estimate \eqref{Eq: energy estimate}.
\end{proof}
\begin{lemma}\label{Lem: parallel spinor}
     Given any constant spinor $\hat s_0\in C^\infty(E,\hat{\mathcal S}|_E)$, there is a spinor $\hat s\in H^1_{loc}(M,\bar g,\hat{\mathcal S})$ such that $\widehat\nabla\hat s=0$ and that $\hat s$ is asymptotically constant to $\hat s_0$ with respect to $\bar g$,
     where $\widehat\nabla$  denotes the induced spinor connection on $\hat{\mathcal S}$ with respect to the metric $\bar g$.
\end{lemma}

\begin{proof}
Fix a real positive sequence $\varepsilon_i\to 0$ as $i\to\infty$ and a constant spinor $\hat s_0\in C^\infty(E,\hat{\mathcal S}|_E)$. From Lemma \ref{Lem: approximation spinor} we can construct spinors $$\hat s_i\in H^1_{loc}(M,g_{\varepsilon_i},\hat{\mathcal S}),$$ 
which is asymptotically constant to $\hat s_0$ with respect to $g_{\varepsilon_i}$, such that we have
$$\widehat{\mathcal D}^{\varepsilon_i}\hat s_i=0\mbox{ in }L^2(M,g_{\varepsilon_i},\hat{\mathcal S}),$$
and also the estimate
$$
\|\widehat\nabla^{\varepsilon_i}\hat s_i\|^2_{L^2(M,g_{\varepsilon_i},\hat{\mathcal S})}\leq \frac{n-1}{2}\omega_{n-1}\cdot m(M,g_{\varepsilon_i},E)\cdot\|\hat s_0\|_\infty^2\to 0\mbox{ as }i\to\infty.$$

Write
$$g_{\varepsilon_i}=\left(\frac{u_{\varepsilon_i}}{u}\right)^{\frac{4}{n-2}}\bar g.$$
From the relation \eqref{Eq: connection relation} we have
\begin{equation}\label{Eq: connection relation application}
    \widehat\nabla^\varepsilon_{(\cdot)}=\widehat\nabla_{(\cdot)}+\frac{1}{n-2}\cdot\hat c\left(\overline\nabla \log\left(\frac{u_{\varepsilon}}{u}\right)\right)\hat c(\cdot)+\frac{1}{n-2}\cdot\bar g\left(\overline\nabla \log\left(\frac{u_{\varepsilon}}{u}\right),\cdot\right),
\end{equation}
where we use $\widehat\nabla_{(\cdot)}$ and $\hat c(\cdot)$ to denote the induced spinor connection and the twisted Clifford multiplication on $\hat{\mathcal S}$ with respect to the metric $\bar g$. 

In the following, we restrict our attention on any fixed smooth open subset $U$ of $M$ such that $U\Delta E$ has compact closure. As before, we fix a smooth cut-off function $\zeta$ on $M$ such that $\zeta\equiv 1$ around the infinity of $E$ and $\zeta \equiv 0$ outside $E$, and consider the spinor
\begin{equation}\label{Eq: define spinor}
    \hat s_i^*:=\hat s_i-\zeta\hat s_0.
\end{equation}
It follows from the expansion of $u$ from Corollary \ref{Cor: function u}, the uniform bound of $u^{-1}$ on $U$, the fact $\varepsilon_i\to 0$ as $i\to+\infty$, and the relation \eqref{Eq: connection relation application} that we can derive
$$\|\widehat\nabla\hat s_i^*\|_{L^2(U,\bar g,\hat{\mathcal S})}^2\leq (1+o(1))\|\widehat\nabla^{\varepsilon_i}\hat s_i^*\|_{L^2(U,g_{\varepsilon_i},\hat{\mathcal S})}^2+o(1)\|\hat s_i^*\|^{2}_{L^{2}_{-q^*}(U,\bar g,\hat{\mathcal S})}$$
as $i\to+\infty$. 
On the other hand, it follows from the weighted Poincar\'e inequality (see \cite[Theorem 1.3]{Bartnik1986} for instance) and the Kato inequality that
$$\|\widehat\nabla\hat s_i^*\|^2_{L^2(U,\bar g,\hat{\mathcal S})}\geq \|\widehat\nabla\widehat{|\hat s_i^*|}\|^2_{L^2(U,\bar g)}\geq C\|\hat s_i^*\|^2_{L^{2}_{-q^*}(U,\bar g,\hat{\mathcal S})}.$$
Therefore, we conclude
$$\|\widehat\nabla\hat s_i^*\|^2_{L^2(U,\bar g,\hat{\mathcal S})}+\|\hat s_i^*\|^2_{L^{2}_{-q^*}(U,\bar g,\hat{\mathcal S})}\leq C\|\widehat\nabla^{\varepsilon_i}\hat s_i^*\|_{L^2(U,g_{\varepsilon_i},\hat{\mathcal S})}$$
for $i$ large enough. Again, using the expansion of $u$ from Corollary \ref{Cor: function u}, the uniform bound of $u^{-1}$ on $U$, the fact $\varepsilon_i\to 0$ as $i\to+\infty$, and the relation \eqref{Eq: connection relation application} we can derive from \eqref{Eq: define spinor} that
\[
\begin{split}\|\widehat\nabla^{\varepsilon_i}\hat s_i^*\|&_{L^2(U,g_{\varepsilon_i},\hat{\mathcal S})}\leq \|\widehat\nabla^{\varepsilon_i}\hat s_i\|_{L^2(U,g_{\varepsilon_i},\hat{\mathcal S})}\\
&\qquad+(1+o(1))\|\widehat\nabla(\zeta\hat s_0)\|_{L^2(U,\bar g,\hat{\mathcal S})}+o(1)\|\overline\nabla u\|_{L^2(U,\bar g)}
\end{split}\]
as $i\to+\infty$.
 In particular, all the spinors $\hat s_i^*$ are locally uniformly bounded in $H^1_{loc}(U,\bar g,\hat{\mathcal S})$. Based on a use of the diagonal argument, we conclude that the spinors $\hat s_i^*$ converge to some limit spinor $\hat s_\infty^*$ weakly in $H^1_{loc}(M,\bar g,\hat{\mathcal S})$ up to a subsequence. 

 Define
 $$\hat s_\infty:=\hat s_\infty^*+\zeta\hat s_0.$$
 Clearly, $\widehat\nabla\hat s_i$ converge to $\widehat\nabla\hat s_\infty$ weakly in $L^2_{loc}(M,\bar g,\hat{\mathcal S})$. A similar analysis as before yields
 \[
 \begin{split}
     \|\widehat\nabla\hat s_i\|_{L^2(U,\bar g,\hat{\mathcal S})}\leq (1&+o(1))\|\widehat\nabla^{\varepsilon_i}\hat s_i\|_{L^2(U,g_{\varepsilon_i},\hat{\mathcal S})}\\&\,\,\,+o(1)\|\hat s_i^*\|_{L^{2}_{-q^*}(U,\bar g,\hat{\mathcal S})}+o(1)\|\overline \nabla u\|_{L^2(U,\bar g)}\to 0
 \end{split}
 \]
 as $i\to+\infty$. Therefore, we have $\widehat\nabla\hat s_\infty=0$. Note that the weak convergence also yields $$\hat s_\infty-\hat s_0\in H^{1}_{-q^*}(E,\bar g,\hat{\mathcal S}).$$ The proof is completed by taking $\hat s=\hat s_\infty$.
\end{proof}

Now we are ready to prove the flatness of $(M,\bar g)$.
\begin{proposition}
$(M,\bar g)$ is flat.
\end{proposition}
\begin{proof}
    Fix a non-zero constant spinor $\hat s_0\in C^\infty(E,\hat{\mathcal S}|_E)$. Using Lemma \ref{Lem: parallel spinor} we can construct a spinor $\hat s_0^*\in H^1_{loc}(M,\bar g,\hat{\mathcal S})$ such that $\widehat\nabla\hat s_0^*=0$ and that $\hat s_0^*$ is asymptotically constant to $\hat s_0$ with respect to $\bar g$. In particular, we have
    \begin{equation}\label{Eq: parallel 1}
        \hat s_0^*-\hat s_0\in L^2_{-q^*}(E,\bar g,\hat{\mathcal S})
    \end{equation}
     Recall that $\{\bar e_i\}$ denotes the $\bar g$-orthonormal frame on $E$ which comes from orthonormalizing the coordinate frame $\{\partial_{x_i}\}$ with respect to the metric $\bar g$. By Lemma \ref{Lem: parallel spinor} we can construct spinors $\hat s_i^*\in H^1_{loc}(M,\bar g,\hat{\mathcal S})$ such that $\widehat\nabla\hat s_i^*=0$ and 
    \begin{equation}\label{Eq: parallel 2}
        \hat s_i^*-\hat c(\bar e_i)\hat s_0\in L^2_{-q^*}(E,\bar g,\hat{\mathcal S}),
    \end{equation}where $\hat c(\cdot)$ denotes the twisted Clifford multiplication on $\hat{\mathcal S}$ with respect to the metric $\bar g$. 
    
    Define the vector fields $V_i$ by
    $$\bar g(V_i,X)=\widehat{\langle \hat c(X)\hat s_0^*,\hat s_i^*\rangle}.$$
    From a direct computation we have
    $$\bar g(\overline\nabla V_i,X)=\widehat{\langle \hat c(X)\widehat\nabla\hat s_0^*,\hat s_i^*\rangle}+\widehat{\langle \hat c(X)\hat s_0^*,\widehat\nabla\hat s_i^*\rangle}=0.$$
    This means that $V_i$ are smooth $\bar g$-parallel vector fields on $M$.

    Now we show that $V_i$ are linearly independent on the whole $M$. Otherwise, we can find real constants $c_i$ such that 
    $$\sum_ic_iV_i\equiv 0\mbox{ but }\sum_ic_i^2\neq 0.$$
    Take $$X=\sum_i c_i\bar e_i.$$
    From \eqref{Eq: parallel 1} and \eqref{Eq: parallel 2} we can derive
    $$0\equiv\sum_i c_i\bar g(V_i,X)=\widehat{\langle \hat c(X)\hat s_0^*,\hat c(X)\hat s_0^*+\hat\eta\rangle}\mbox{ where }\hat\eta\in L^2_{-q^*}(E,\bar g,\hat{\mathcal S}).$$
    In particular, we have
    $$\widehat{|\hat\eta|}\geq \widehat{|\hat c(X)\hat s_0^*|}=\left(\sum_ic_i^2\right)^{\frac{1}{2}}\widehat{|\hat s_0^*|}.$$
    Then we can derive 
    $$\|\hat\eta\|_{L^2_{-q^*}(E,\bar g,\hat{\mathcal S})}^2\geq\left(\sum_ic_i^2\right)\widehat{|\hat s_0^*|}^2\int_E \rho^{-2}\,\mathrm d\mu_{\bar g}=+\infty,$$
 which leads to a contradiction.

    From previous discussion, we know that there exists a global $\bar g$-parallel orthonormal frame on $M$. Therefore, $(M,\bar g)$ is flat.
\end{proof}
\begin{corollary}\label{Cor: scalar flat}
    $(M,g)$ is scalar-flat.
\end{corollary}
\begin{proof}
    Since $(M,\bar g)$ is flat, we have
    $$R_{\bar g}=u^{-\frac{n+2}{n-2}}\left(-\Delta_gu+\frac{n-2}{4(n-1)}R_gu\right)=0.$$
    Recall that $u$ is a positive harmonic function on $(M,g)$ and this yields $R_g=0$.
\end{proof}

\subsection{The topology}
In this subsection, we will make an analysis on the topology of $M$ in the rigidity case.

 \begin{lemma}
 $M$ is simply-connected.
 \end{lemma}
 \begin{proof}
 Suppose that $M$ is not simply-connected. We consider the universal cover $(\hat M,\hat g)$ of $(M,g)$. Since $E$ is simply-connected, there are at least two copies of $E$ in $\hat M$, denoted by $\hat E_1$ and $\hat E_2$ respectively. Note that $(\hat M,\hat g,\hat E_1)$ is a generalized asymptotically flat manifold. It follows from Corollary \ref{Cor: function u} we can find a smooth positive harmonic function $\hat u$ on $(\hat M,\hat g)$ with the expansion
 $$\hat u(\hat x)=1-\mathfrak c(\hat M,\hat g,\hat E_1)\cdot|\hat x|^{2-n}+\hat w_1\mbox{ as }\hat x\to \widehat\infty_1,$$
 where $\widehat\infty_1$ denotes the infinity of $\hat E_1$ and $\hat w_1=O_2(|\hat x|^{2-n-\tau})$ as $\hat x\to\widehat\infty_1$. In a similar way, we can find a positive harmonic function $u$ on $(M, g)$ with the expansion
 $$ u(x)=1-\mathfrak c( M, g, E)\cdot |x|^{2-n}+w\mbox{ as }x\to \infty.$$
 Let $\tilde u$ denote the lift of $u$ on $(\hat M,\hat g)$. Recall that the function $\hat u$ is constructed as a limit of harmonic functions vanishing on inner boundary and converging to one at the infinity of $\hat E_1$. It follows from the maximum principle that we have $\hat u\leq \tilde u$. Using the asymptotical flatness of $(\hat E_2,\hat g)$ we are able to show $\hat u(\hat x)\to 0$ as $\hat x\to \widehat \infty_2$, and so we have $\hat u<\tilde u$ on $\hat M$. This yields
 $$2\mathfrak c(\hat M,\hat g,\hat E_1)>2\mathfrak c(M,g,E)=m(M,g,E)=m(\hat M,\hat g,\hat E_1),$$
 and we obtain a contradiction to Proposition \ref{Prop: mass capacity}.
 \end{proof}

\begin{proposition}\label{Prop: underlying space}
    There is a bounded closed subset $ S$ of $\mathbb R^n$ with Hausdorff dimension no greater than $(n-2)/2$ such that $(M,\bar g)$ is isometric to $\mathbb R^n\setminus S$.
\end{proposition}
\begin{proof}
    Since $(M,\bar g)$ is simply-connected and flat, we are able to construct an isometric immersion
        $$\Psi:(M,\bar g)\to \mathbb R^n.$$ Composed with the stereographic projection, we finally obtain a conformal map from $(M,g)$ to $\mathbb S^n$, which must be an embedding due to the Liouville theorem \cite{Schoen-Yau,Chodosh-Li}. In particular, $\Psi$ is an isometric embedding. Denote $$S=\mathbb R^n\setminus\Psi(M).$$ We are going to show the desired properties of $S$.

    It is clear that $S$ is a closed subset of $\mathbb R^n$  since the isometric embedding $\Psi$ is an open map. To see that $S$ is bounded, notice that $\Psi(E)$ is a smooth region of $\mathbb R^n$ with $\partial \Psi(E)=\Psi(\partial E)$, which has infinite volume. This implies that $\Psi(E)$ is the unbounded component of $\mathbb R^n\setminus \Psi(\partial E)$ and so $S$ is bounded.

    Furthermore, it follows from the characterization result by Karakhanyan \cite{Karakhanyan} for domains on $n$-sphere conformal to a complete scalar-flat manifold that the $(1+2/n, n/2)$-Bessel capacity of $S$ is zero. Consequently, its Hausdorff dimension is at most $(n-2)/2$.
\end{proof}

\begin{corollary}\label{Cor: vanishing}
We have
$$\pi_i(M)=0\mbox{ for all }1\leq i\leq n-1-\left\lfloor\frac{n}{2}\right\rfloor.$$
\end{corollary}
\begin{proof}
From Proposition \ref{Prop: underlying space} it suffices to show that
$$\pi_i(\mathbb{R}^n \setminus S) = 0\mbox{ for all }1 \leq i \leq n - 1 - \left\lfloor \frac{n}{2} \right\rfloor.$$
Given any continuous function $$f_0:\mathbb S^i\to \mathbb R^n\setminus S,$$ it follows from the density of smooth functions in continuous functions that we can find a smooth map 
$$f: \mathbb S^i \to \mathbb{R}^n \setminus S$$ representing the homotopy class $[f_0]$.  Since $\pi_i(\mathbb{R}^n) = 0$, there exists a smooth map
$$F:\bar{\mathbb B}^{i+1} \to \mathbb{R}^n\mbox{ with }F|_{\partial \mathbb B^{i+1} = \mathbb S^i} = f.$$

Now let us apply a transversality argument to show that the map $F$ can be perturbed such that its image lies in $\mathbb R^n\setminus S$. For convenience, given any constant $0<s<1$, let $\mathbb B^{i+1}_s$ denote the Euclidean ball with radius $s$ and $\mathcal A_s$ denotes the annulus $\bar{\mathbb B}^{i+1}\setminus {\mathbb B}^{i+1}_s$. From the continuity of the map $F$ and the closedness of $S$, we can choose a small constant $\varepsilon > 0$ such that 
    \[
    F(\mathcal A_{1-2\varepsilon}) \cap \mathcal{N}_\varepsilon(S) = \emptyset,
    \]
    where $\mathcal{N}_\varepsilon(S)$ denotes the $\varepsilon$-neighborhood of $S$ in $\mathbb R^n$. 
    
    Let $\eta \in C_c^\infty(\mathbb B^{i+1})$ be a cut-off function with $\eta \equiv 1$ on $\mathbb B_{1-\varepsilon}^{i+1}$ and $\eta\equiv 0$ around $\partial \mathbb B^{i+1}$. We consider the following smooth map
    \[
    H: \mathbb B^{i+1} \times \mathbb{R}^n \to \mathbb{R}^n, \, (x, v) \mapsto F(x) + \eta(x) v.
    \]
    Clearly, there is an open neighborhood $U$ of the origin in $\mathbb R^n$ such that 
    \begin{equation}\label{Eq: boundary no intersection}
        H(\mathcal A_{1-2\varepsilon}\times U)\cap S=\emptyset.
    \end{equation}
    Notice that $H$ is a submersion restricted to $\mathbb B^{i+1}_{1-\varepsilon}\times U$ since it comes from translations of $\mathbb R^n$. Denote
    $$T = H^{-1}(S) \cap (\mathbb B^{i+1}_{1-\varepsilon} \times U).$$
    Since submersions are locally canonical projection maps between Euclidean spaces, we can derive
    $$\dim_H(T)\leq \dim_H(S)+i+1.$$
    Let $\pi_2:\mathbb B^{i+1}\times \mathbb R^n\to \mathbb R^n$ denote the projection map. Then we have $$\dim_H(\pi_2(T))\leq \dim_H(T)\leq \dim_H(S)+i+1<n.$$
    From this we can take some $v_0\in U$ such that $v_0\notin\pi_2(T)$, which means
    \begin{equation}\label{Eq: interior no intersection}
        H^{-1}(S)\cap (\mathbb B^{i+1}_{1-\varepsilon}\times\{v_0\})=\emptyset. 
    \end{equation}

    Define 
    $$\tilde F(x) = F(x) + \eta(x) v_0.$$
    From \eqref{Eq: boundary no intersection}-\eqref{Eq: interior no intersection} we know that $\tilde F$ maps $\mathbb B^{i+1}$ into $\mathbb R^n\setminus S$ with $\tilde F|_{\partial \mathbb B^{i+1}=\mathbb S^i}=f$. This shows that $[f_0]=0$ in $\pi_i(\mathbb R^n\setminus S)$ and we complete the proof.
\end{proof}
 \subsection{Proof of Theorem \ref{Thm: main1}} 
    \begin{proof}[Proof of Theorem \ref{Thm: main1}]
    It follows from Proposition \ref{Prop: mass capacity}, Proposition \ref{Prop: underlying space} and Corollary \ref{Cor: vanishing}.
\end{proof}

\section{Modifications in corner case}\label{Sec: corner}

\subsection{Spin positive mass theorem with corner}
First we show that the spin positive mass theorem can be generalized for generalized asymptotically flat manifolds with corner.
\begin{proposition}\label{Prop: spin PMT with corner}
    Assume that $(M^n,g,E)$ is a spin generalized asymptotically flat manifold with corner, which has nonnegative scalar curvature. Assume that the sum of the mean curvatures of the corner on both sides with respect to the outward unit normal is nonnegative.  Then we have
    $$m(M,g,E)\geq 0,$$
    where the equality holds if and only if $(M,g)$ is smooth up to a change of the smooth structure around the corner and isometric to the Euclidean space.
\end{proposition}
\begin{proof}
 Notice that Lipschitz metrics are sufficient to carry out analysis based on weighted Sobolev spaces $W^{1,p}_\delta(U,g,\hat{\mathcal S})$. In particular, Lemma \ref{Lem: nonnegative index} is still true in the corner case. On the other hand,
by following the computation from \cite[page 177]{DanLeeGeometricRelativity}, we can obtain the modified mass formula
      \[
    \begin{split}
        \frac{n-1}{2}&\omega_{n-1}\cdot m(M,g,E)\cdot\|\hat s\|_{\infty}^2+\|\widehat B_\psi\hat s\|^2_{L^2(U,g,\hat{\mathcal S})}\\
        &\geq  \|\widehat\nabla\hat s\|^2_{L^2(U,g,\hat{\mathcal S})}+\int_U\theta_\psi\widehat{|\hat s|}^2\mathrm d\mu_{g}+\int_{\partial U}\eta_\psi \widehat{|\hat s|}^2\mathrm d\sigma_{g}+\int_\Sigma\beta\widehat{|\hat s|}^2\mathrm d\sigma_{g},
    \end{split}
    \]
    where $\beta$ denotes the sum of the mean curvatures of the corner $\Sigma$ on both sides with respect to the outward unit normal. By following the proof of Lemma \ref{Lem: approximation spinor} we can construct a harmonic spinor $\hat s$ asymptotically constant to $\hat s_0\neq 0$ satisfying
      \[
    \begin{split}
        \frac{n-1}{2}&\omega_{n-1}\cdot m(M,g,E)\cdot\|\hat s\|_{\infty}^2\geq  \|\widehat\nabla\hat s\|^2_{L^2(U,g,\hat{\mathcal S})}+\int_\Sigma\beta\widehat{|\hat s|}^2\mathrm d\sigma_{g}.
    \end{split}
    \]
   Recall that $\beta$ is nonnegative along $\Sigma$. Therefore, we obtain $m(M,g,E)\geq 0$, where the equality implies that $\hat s$ is a parallel spinor.
   
   In the following, we assume $m(M,g,E)=0$. As before, we can construct sufficiently many parallel spinors and use them to further construct a parallel orthonormal frame $\{V_i\}$ on $(M,g)$. From \cite[Lemma 3.3]{ST2002} we know that the parallel spinors are continuous through the corner $\Sigma$ and so does the frame $\{V_i\}$. Since $(M,g)$ is piecewise smooth, the frame $\{V_i\}$ is piecewise smooth and globally Lipschitz. Denote the dual coframe of $\{V_i\}$ by $\{\omega_i\}$. Then $\{\omega_i\}$ is also piecewise smooth and globally Lipschitz. Since the dual 1-forms $\omega_i$ are parallel, they are closed $1$-forms. From the Poincar\'e lemma we can find piecewise smooth $C^{1,1}$-functions $f_i$ such that $\mathrm df_i=\omega_i$ around any point $p$ on the corner $\Sigma$. Define $F=(f_1,f_2,\ldots,f_n)$. Then the map $F$ gives a piecewise smooth $C^{1,1}$-isometry 
   $$F:(U,g)\to F(U)\subset \mathbb R^n,$$ where $U$ is an open neighborhood of $p$.

   Notice that $M$ can be given a new smooth structure by viewing it as a smooth manifold coming from the gluing of the metric completion of $M\setminus \Sigma$ along $\Sigma$, where the collar neighborhoods are chosen to be Fermi coordinate charts with respect to the metric $g$. Let us show that the map $F$ is smooth under this new choice of smooth structure. Notice that the map $F:(U,g)\to F(U)\subset \mathbb R^n$ remains a piecewise smooth $C^{1,1}$-isometry. Since $\Sigma$ is separating, we can pick a global unit normal vector field $\nu$ on $\Sigma$. Fix a Fermi coordinate chart
   $$\Phi:V\times (-\varepsilon,\varepsilon)\to M,\,(q,s)\mapsto \exp_q(s\nu),$$
   where $V$ is an open neighborhood of $p$ on $\Sigma$.
   Using the metric expression in the Fermi coordinate,  we see that if the neighborhood $V$ and the constant $\varepsilon$ are small enough, then there is a unique unit-speed minimizing geodesic 
   $$\gamma:[-\varepsilon,\varepsilon]\to(U,g),\,s\mapsto \exp_q(s\nu),$$
   connecting $\exp_q(-\varepsilon\nu)$ to $\exp_q(\varepsilon\nu)$ for any $q\in V$. Since any pair of points in the Euclidean space determine a unique unit-speed minimizing geodesic, we conclude that $F\circ\gamma$ is the unique unit-speed minimizing geodesic connecting $F(\gamma(-\varepsilon))$ and $F(\gamma(\varepsilon))$ given by
   \begin{equation}\label{Eq: line connection}
       (F\circ\gamma)(s)=(F\circ\gamma)(-\epsilon)+(s+\epsilon)F_*(\gamma'(-\epsilon)).
   \end{equation}
   Denote 
   $$\Sigma_{-\varepsilon}=\{\exp_{q}(-\varepsilon\nu):q\in V\}\mbox{ and }S_{-\epsilon}=F(\Sigma_{-\epsilon}).$$
   Let $\nu_{-\varepsilon}$ denote the unit normal vector field on $\Sigma_{-\varepsilon}$ pointing to $\Sigma$. Recall that $F$ is smooth away from the corner. Therefore, $S_{-\varepsilon}$ is a smooth hypersurface in the Euclidean space and $F_*(\nu_{-\varepsilon})$ is a smooth unit normal vector field on $S_{-\varepsilon}$. From the previous discussion and the characterization for focal radius, we see that the map
   $$\Psi:S_{-\varepsilon}\times [0,2\varepsilon)\to \mathbb R^n,\,(x,\tau)\mapsto x+\tau F_*(\nu_{-\varepsilon}),$$
   is smooth. From \eqref{Eq: line connection} we can write
   $$(F\circ\Phi)(q,s)=\Psi(F(\exp_q(-\varepsilon\nu)),s+\varepsilon),$$
   which is a composition of smooth maps. In particular, the map $F$ is smooth around $p$ under the new choice of smooth structure. 

   From above discussion we know that $g$ turns out to be a smooth flat metric on $M$ with respect to the new smooth structure. The volume comparison theorem and the asymptotical flatness of $g$ yield that $(M,g)$ is isometric to the Euclidean space.
\end{proof}

\subsection{Proof of Theorem \ref{Thm: main2} and Corollary \ref{Cor: Bray}}

\begin{proof}[Proof of Theorem \ref{Thm: main2} ]
    By repeating the proof of Corollary \ref{Cor: function u} we are able to construct a positive weakly harmonic $C^{1,\alpha}$-function $u$ on $(M,g)$ such that we have the expansion 
    $$u(x)=1-\mathfrak c(M,g,E)\cdot|x|^{2-n}+w,$$
    with $w=O_2(|x|^{2-n-\tau})$ as $x\to\infty$. For any constant $\varepsilon>0$, the conformal manifold
    $(M,g_\varepsilon,E)$ with $$g_\varepsilon=\left(\frac{u+\varepsilon}{1+\varepsilon}\right)^{\frac{4}{n-2}}g$$ is a spin generalized asymptotically flat manifold with corner satisfying the assumption of Proposition \ref{Prop: spin PMT with corner}. As a consequence, we have
    $$0\leq m(M,g_\varepsilon,E)=m(M,g,E)-2(1+\varepsilon)^{-1}\mathfrak c(M,g,E).$$
    Letting $\varepsilon\to 0$, we obtain the desired mass-capacity inequality
    $$m(M,g,E)\geq 2\mathfrak c(M,g,E).$$
    If the equality holds, then the same argument in the proof of Proposition \ref{Prop: spin PMT with corner} yields that $(M,\bar g)$ with $\bar g=u^{\frac{4}{n-2}}g$ is smooth and flat up to a change of the smooth structure. Notice that $u^{-1}$ is a weakly harmonic function on $(M,\bar g)$. We conclude that $u$ is smooth and so $g$ is also smooth. Now we return to the smooth case and the desired conclusions come from Theorem \ref{Thm: main1}.
\end{proof}

\begin{proof}[Proof of Corollary \ref{Cor: Bray} directly from Theorem \ref{Thm: main2}]
    Let $(\widehat M,\widehat g)$ denote the metric double of $(M,g)$ and $\widehat E$ denote a lift of $E$ in $\widehat M$. Then $(\widehat M,\widehat g,\widehat E)$ is a spin generalized asymptotically flat manifold with corner satisfying all the assumption of Theorem \ref{Thm: main2}. In particular, we have
    $$m(M,g,E)=m(\widehat M,\widehat g,\widehat E)\geq 2\mathfrak c(\widehat M,\widehat g,\widehat E).$$
    In the following, let us show
    \begin{equation}\label{Eq: capacity relations}
        \mathfrak c(\widehat M,\widehat g,\widehat E)= \frac{1}{2}m(M,g,\Sigma).
    \end{equation}
    By minimizing the energy we can construct a harmonic function $u$ on $(M,g)$ with $u=0$ on $\Sigma$ and $u(x)\to 1$ as $x\to\infty$ such that
    $$\frac{1}{n(n-2)\omega_n}\int_M|\nabla_gu|^2\,\mathrm d\mu_g=\mathfrak 
c(M,g,\Sigma).$$
Let $\widehat M_+$ denote the copy of $M$ in $\widehat M$ containing the end $\widehat E$ and $\widehat M_-$ denote another copy. Let $\pi:\widehat M\to M$ denote the canonical projection map. Define
\[
\hat u(\hat x)=
\begin{cases}
\,\,u(\pi(x)), & \hat x\in \widehat M_+,\\
-u(\pi(\hat x)), & \hat x\in \widehat M_-.
\end{cases}
\]
Then $(\hat u+1)/2$ is the harmonic function attaining the capacity $\mathfrak c(\widehat M,\widehat g,\widehat E)$. This gives \eqref{Eq: capacity relations} and so we have
$$m(M,g,E)\geq \mathfrak c(M,g,\Sigma).$$

If the equality holds, then $(\widehat M,\hat g,\widehat E)$ also attains the equality of the mass-capacity inequality. So $(\widehat M,\hat g)$ is harmonically conformal to an open subset $U\subset\mathbb R^n$ with $\mathbb R^n\setminus U$ bounded and closed. The $\mathbb Z_2$-isometry of $(\widehat M,\hat g)$ induces a conformal involution $\rho:U\to U$. Let $\widehat \Sigma$ denote the intersection of $\widehat M_+$ and $\widehat M_-$. Then $\widehat \Sigma$ is embedded into $\mathbb R^n$ as a closed smooth hypersurface $T$, which consists of fixed points of $\rho$. Recall that   the Liouville theorem \cite[Theorem 5.2]{dC} yields that $\rho$ is a composition of isometries, dilatations or inversions, at most one of each. In the following, we divide the discussion into two cases:

If there is no inversion, then there is also no dilatation since $\rho^2$ is the identity and the volume element should be kept. Therefore, $\rho$ is simply an isometry. Recall that $\rho$ is a non-trivial involution, so we can find two points $p\neq q$ such that $\rho$ maps $p$ to $q$ and vice versa. In particular, $\rho$ keeps the middle point of $p$ and $q$, which we take as the origin up to a change of coordinate.
Let $P$ denote the perpendicular bisector of points $p$ and $q$. Then $\rho$ is the direct sum of reflection along the line $\overline{pq}$ and an involution on $P$. By induction, we can decompose $\mathbb R^n=V\oplus W$ such that $\rho=\rho_V\oplus\rho_W$ where $\rho_V$ is the antipodal map and $\rho_W$ is the identity. In particular, either $\rho$ has non-compact fixed points or it has a unique fixed point, which leads to a contradiction. 

If there is an inversion, then we know that $\widehat M_+$ and $\widehat M_-$ are mapped to the outer region of $T$ and the inner region of $T$ with a single point $q_*$ removed up to an order. In particular, $(\widehat M,\hat g)$ is harmonically conformal to $\mathbb R^n\setminus\{q_*\}$. Since any positive harmonic function on $\mathbb R^n\setminus\{q_*\}$ has the form
$$\alpha+\beta |x-q_*|^{2-n},$$
where $\alpha$ and $\beta$ are two nonnegative constants not equal to zero at the same time. Since $(\widehat M,\hat g)$ is complete, both $\alpha$ and $\beta$ for the harmonic conformal factor are positive. Then $(\widehat M,\hat g)$ is isometric to the whole spatial Schwarzschild manifold and $\widehat\Sigma$ is the unique minimal hypersurface. This yields that $(M,g)$ is isometric to the half spatial Schwarzschild manifold.
\end{proof}


\appendix
\section{Some computations in spin geometry}\label{App: A}
Let $(M^n,g)$ be a Riemannian $n$-manifold. Fix an orthonormal frame $\{e_i\}$ with respect to the metric $g$ and define the connection $1$-form
$$\omega_{ij}=g(\nabla e_i,e_j),$$
where $\nabla$ is the Levi-Civita connection of the metric $g$.
\begin{lemma}\label{Lem: connection form}
    We have
    \[
    \begin{split}
    2\omega_{ij}(X)=&e_i g(X,e_j)-e_jg(X,e_i)\\
    &+g([X,e_i],e_j)-g([X,e_j],e_i)-g([e_i,e_j],X).
    \end{split}\]
\end{lemma}
\begin{proof}
    This follows directly from the Koszul formula.
\end{proof}
Let $\tilde g=e^{2u}g$ be a conformal metric of $g$. Define $\tilde e_i=e^{-u}e_i$. Then $\{\tilde e_i\}$ is an orthonormal frame with respect to the metric $\tilde g$. Correspondingly, we can define the connection $1$-form
$$\tilde \omega_{ij}=\tilde g(\tilde \nabla \tilde e_i,\tilde e_j),$$
where $\tilde\nabla$ denotes the Levi-Civita connection of the metric $\tilde g$.
\begin{lemma}\label{Lem: conformal connection}
    We have 
    $$\tilde \omega_{ij}(X)=\omega_{ij}(X)+e_i(u)g(X,e_j)-e_j(u)g(X,e_i).$$
\end{lemma}
\begin{proof}
    From Lemma \ref{Lem: connection form} we have
    \[
    \begin{split}
    2\tilde\omega_{ij}(X)=&\tilde e_i \tilde g(X,\tilde e_j)-\tilde e_j\tilde g(X,\tilde e_i)\\
    &+\tilde g([X,\tilde e_i],\tilde e_j)-\tilde g([X,\tilde e_j],\tilde e_i)-\tilde g([\tilde e_i,\tilde e_j],X).
    \end{split}\]
    Using the facts $\tilde g=e^{2u}g$ and $\tilde e_i=e^{-u}e_i$, we can compute
    $$\tilde e_i\tilde g(X,\tilde e_j)=e_i(u)g(X,e_j)+e_i g(X,e_j),$$
    $$\tilde e_j\tilde g(X,\tilde e_i)=e_j(u)g(X,e_i)+e_j g(X,e_i),$$
    $$\tilde g([X,\tilde e_i],\tilde e_j)=-X(u)g(e_i,e_j)+g([X,e_i],e_j),$$
    $$\tilde g([X,\tilde e_j],\tilde e_i)=-X(u)g(e_j,e_i)+g([X,e_j],e_i),$$
    and
    $$\tilde g([\tilde e_i,\tilde e_j],X)=-e_i(u)g(X,e_j)+e_j(u)g(X,e_i)+g([e_i,e_j],X).$$
    Then we have
    $$2\tilde \omega_{ij}(X)=2\omega_{ij}(X)+2(e_i(u)g(X,e_j)-e_j(u)g(X,e_i)).$$
    This completes the proof.
\end{proof}
Now we assume that $M$ is spin. Then we can take a spin structure 
$$\pi:P_{spin}\to P_{SO}(M,g).$$
Let $\rho:spin(n)\to End(\Delta_n)$ be the complex spin representation, and the spinor bundle $\mathcal S$ is defined to be $\mathcal S=P_{spin}\times_\rho \Delta_n$. Recall that we have the Clifford multiplication
$$c:Cl(T_xM,g_x)\to End(\mathcal S_x).$$
With the spin structure the Levi-Civita connection of the metric $g$ naturally induces a spinor connection on the spinor bundle $\mathcal S$, still denoted by $\nabla$.
\begin{lemma}\label{Lem: derivative from connection}
    We have
    $$\nabla_Xs=\frac{1}{4}\omega_{ij}(X)c(e_i)c(e_j)s.$$
\end{lemma}
\begin{proof}
    This follows from the Lie algebra isomorphism
    $$\mathrm{ad}:\mathfrak{spin}_n\to \mathfrak{so}_n,\,\frac{1}{4}A_{ij}e_ie_j\mapsto (A_{ij})_{n\times n},$$
    where $(A_{ij})_{n\times n}$ is any skew-symmetric matrix.
\end{proof}
Recall that there is a Hermitian metric $\langle\cdot,\cdot\rangle$ on the spinor bundle $\mathcal S$ such that
\begin{itemize}
    \item for any $g$-unit vector $e$ we have
    $$\langle c(e)s_1,c(e)s_2\rangle=\langle s_1,s_2\rangle,$$
    \item and we have
    $$X\langle s_1,s_2\rangle=\langle \nabla_Xs_1,s_2\rangle+\langle s_1,\nabla_Xs_2 \rangle.$$
\end{itemize}
Notice that the frame bundles $P_{SO}(M,g)$ and $P_{SO}(M,\tilde g)$ are isomorphic through the following natural correspondence
$$\{e_i\}\in P_{SO}(M,g)\leftrightarrow \left\{\tilde e_i=\frac{e_i}{|e_i|_{\tilde g}}\right\}\in P_{SO}(M,\tilde g).$$
Therefore, we can take the same spin structure $\tilde \pi:P_{spin}\to P_{SO}(M,\tilde g)$ and the spinor bundle $\mathcal S$. The Clifford multiplication will be denoted by 
$$\tilde c:Cl(T_xM,\tilde g_x)\to End(\mathcal S_x).$$
\begin{lemma}\label{lem:conformal clifford representation}
    We have
    $$\tilde c(\tilde e_i)=c(e_i).$$
    In particular, for any $\tilde g$-unit vector $\tilde e$ we have 
    $$\langle \tilde c(\tilde e)s_1,\tilde c(\tilde e)s_2\rangle=\langle s_1,s_2\rangle.$$
\end{lemma}
\begin{proof}
    Clearly we can take $p\in P_{spin}$ such that $\pi(p)=\{e_i\}$ and $\tilde \pi(p)=\{\tilde e_i\}$. Then for any $v\in \mathbb R^n$ and $s=(p,\varphi)\in \mathcal S$ with $\varphi\in \Delta_n$ we have
    $$\tilde c(v^i\tilde e_i)s=(p,v\cdot \varphi)=c(v^ie_i)s.$$
    This completes the proof.
\end{proof}
Let $\tilde \nabla$ denote the Levi-Civita connection of $(M,\tilde g)$ as well as the induced spinor connection on the spinor bundle $\mathcal S$.
\begin{lemma}
    We have
    $$\tilde \nabla_Xs=\nabla_Xs+\frac{1}{2}c(\nabla u)c(X)s+\frac{1}{2}g(\nabla u,X)s.$$
\end{lemma}
\begin{proof}
    Using Lemma \ref{Lem: conformal connection}, Lemma \ref{Lem: derivative from connection} and Lemma \ref{lem:conformal clifford representation} we have
    \[
    \begin{split}
        \tilde \nabla_Xs&=\frac{1}{4}\tilde \omega_{ij}(X)\tilde c(\tilde e_i)\tilde c(e_j)s\\
        &=\frac{1}{4}\omega_{ij}(X)c(e_i)c(e_j)s+\frac{1}{4}[e_i(u)g(X,e_j)-e_j(u)g(X,e_i)]c(e_i)c(e_j)s\\
        &=\nabla_Xs+\frac{1}{2}\sum_{i\neq j}e_i(u)g(X,e_j)c(e_i)c(e_j)s.
    \end{split}
    \]
    The last term equals
    $$\frac{1}{2}c(\nabla u)c(X)s+\frac{1}{2}g(\nabla u,X)s,$$
    and we complete the proof.
\end{proof}
\begin{corollary}
    We have
    $$X\langle s_1,s_2\rangle=\langle \tilde\nabla_Xs_1,s_2\rangle+\langle s_1,\tilde\nabla_Xs_2 \rangle.$$
\end{corollary}
\begin{proof}
    This follows directly from
    $$\langle c(\nabla u)c(X)s_1,s_2\rangle =\langle s_1,c(X)c(\nabla u)s_2\rangle$$
    and
    $$c(\nabla u)c(X)+c(X)c(\nabla u)+2g(\nabla u,X)=0.$$
We omit the computational details.
\end{proof}

\printbibliography
\end{document}